\documentclass[12pt]{amsart}
\usepackage[colorlinks]{hyperref}
\usepackage{fullpage}
\usepackage{stmaryrd}
\usepackage{graphicx}
\usepackage{hyperref}
\usepackage{amsmath,amssymb,amsfonts,mathrsfs,amscd}
\usepackage{newtxtext,newtxmath}
\usepackage[all,cmtip]{xy}

\newtheorem{thm}{Theorem}[section]
\newtheorem{lem}[thm]{Lemma}
\newtheorem{cor}[thm]{Corollary}

\newtheorem{thmA}{Theorem}

\theoremstyle{definition}
\newtheorem{defn}[thm]{Definition}

\theoremstyle{remark}

\numberwithin{equation}{section}

\def \X {{\rm X}_{\pi}}
\def \Y {{\rm X}_{\pi'}}

\DeclareMathOperator{\B}{{\rm B}_{\pi}}

\newcommand{\BP}{{\rm B}_{\pi}^*}

\newcommand{\NM}{\vartriangleleft}

\newcommand{\NN}{\mathcal{N}}
\newcommand{\bN}{\mathbf{N}}

\DeclareMathOperator{\Char}{Char}
\DeclareMathOperator{\Irr}{Irr}

\DeclareMathOperator{\Lin}{Lin}

\DeclareMathOperator{\Syl}{Syl}
\DeclareMathOperator{\Hall}{Hall}

\begin{document}
\title{A new generalization of the McKay conjecture for $p$-solvable groups\\
\medskip
\small \emph{To the memory of I. M. Isaacs}}

\author{Huimin Chang}

\address{School of Mathematical and Statistics, Shanxi University, Taiyuan, 030006, China.}

\email{changhuimin@sxu.edu.cn}

\email{jinping@sxu.edu.cn}

\author{Ping Jin*}

\thanks{*Corresponding author}

\keywords{$p$-solvable group, normal $p$-series,
self-stabilizing pairs, $p$-special character, the McKay conjecture}

\date{}

\maketitle

\begin{abstract}
Let $P$ be a Sylow $p$-subgroup of a finite $p$-solvable group $G$,
where $p$ is a prime.
Using a normal $p$-series $\NN$ of $G$,
we introduce the notion of $(\NN,p)$-stable characters
and prove that $G$ and ${\bf N}_G(P)$ have equal numbers of
such characters, which gives a new generalization of the McKay conjecture for $p$-solvable groups. Also, we establish a canonical bijection between these characters
in the case where $G$ has odd order.
Our proofs depend heavily on the theory of self-stabilizing pairs
founded by M. L. Lewis, as well as some results of $\pi$-special characters
due to I. M. Isaacs.
\end{abstract}

\section{Introduction}
 Fix a prime $p$, and let $G$ be a finite group with $P\in\Syl_p(G)$.
Then the McKay conjecture, which is recently confirmed in \cite{CS},
asserts that $G$ and $\bN_G(P)$ have equal numbers of irreducible characters with
degrees not divisible by the prime $p$, i.e.,
$|\Irr_{p'}(G)|=|\Irr_{p'}(\bN_G(P))|$.
For various refinements of the McKay conjecture, the reader is referred to \cite{N2018},
and in this paper, we will present a new generalization of this conjecture for $p$-solvable groups.

We introduce a temporary notation. If $X$ is a finite group, we write
$$\Omega_i(X)=\{\chi\in\Irr(X)|\chi(1)_p=p^i\}$$
for each integer $i\ge 0$.
Then the McKay conjecture
can be restated as $|\Omega_0(G)|=|\Omega_0(\bN_G(P))|$.
In order to generalize the McKay conjecture, therefore,
it is natural to ask for each $i\ge 1$,
what kind of connection exists between $\Omega_i(G)$
and $\Omega_i(\bN_G(P))$.
By using the theory of self-stabilizing pairs founded by Lewis in
\cite{L2006,L2006a,L2008,L2009,L2010},
we are able to establish a new quantitative relationship between these two sets
in the case where $G$ is $p$-solvable.

Specifically, given a $p$-solvable group $G$,
we fix a {\bf normal $p$-series} $\NN=\{N_i\}$ of length $n$ of $G$, that is,
$$1=N_0\le N_1\le \cdots \le N_n=G,$$
where $N_i\NM G$ and $N_i/N_{i-1}$ is either a $p$-group or a $p'$-group
for $i=1,\ldots,n$.
We say that $\chi\in\Irr(G)$ is {\bf $(\NN,p)$-stable}
if $N_i$ fixes an irreducible constituent of $\chi_{N_{i-1}}$
whenever $N_i/N_{i-1}$ is a $p$-group,
and we use $\Irr(G,\NN)$ to denote the set of $(\NN,p)$-stable characters of $G$.
More generally, if $H$ is a subgroup of $G$,
then $\NN_H=\{H\cap N_i\}$ is clearly a normal $p$-series of $H$,
and thus we can define the $(\NN_H,p)$-stability of irreducible characters of $H$.
That is, $\theta\in\Irr(H)$ is $(\NN_H,p)$-stable
if $H\cap N_i$ fixes an irreducible constituent of $\theta_{H\cap N_{i-1}}$
whenever $N_i/N_{i-1}$ is a $p$-group.
For natational convenience, we will always use $\Irr(H,\mathcal{N})$ instead of $\Irr(H,\NN_H)$.
Note that if $H$ is a Sylow $p$-subgroup of $G$,
then $\Irr(H,\NN)$ consists of those characters $\alpha\in\Irr(H)$
such that $\alpha_{H\cap N}$ is homogeneous
(i.e., a multiple of an irreducible character) for each $N\in\NN$.
Furthermore, for each $i\ge 0$ we write $$\Irr_i(H,\mathcal{N})=\{\theta\in\Irr(H,\mathcal{N})|\theta(1)_p=p^i\},$$
which clearly is contained in $\Omega_i(H)$,
so that $\Irr(H,\NN)=\bigsqcup_{i\ge 0}\Irr_i(H,\NN)$.
(From here on, we use the symbol $\bigsqcup$ to denote the disjoint union of sets.)
In this way, we give a natural graded structure on the set $\Irr(H,\mathcal{N})$
by using the normal $p$-series $\NN$.

Now our main result can be stated as follows,
which generalizes the McKay conjecture for $p$-solvable groups
with an independent proof.

\begin{thmA}\label{A}
Let $G$ be a $p$-solvable group with $P\in\Syl_p(G)$, and let $\NN$ be a
normal $p$-series of $G$. Then the following hold.

{\rm (1)} $\Irr_0(G,\NN)=\Irr_{p'}(G)$ and
$\Irr_0({\bf N}_G(P),\NN)=\Irr_{p'}({\bf N}_G(P))$.

{\rm (2)} $|\Irr_i(G,\NN)|=|\Irr_i({\bf N}_G(P),\NN)|$ for all $i\ge 0$.

{\rm (3)} For $i\ge 0$, $\Irr_i(G,\NN)$ is nonempty
if and only if $p^i=\alpha(1)$ for some  $\alpha\in\Irr(P,\NN)$.

{\rm (4)} $|\Irr(G,\NN)|=|\Irr({\bf N}_G(P),\NN)|$.
\end{thmA}

We mention that there exist many $p$-solvable groups $G$
such that the set $\Irr_i(G,\NN)$ is nonempty for some integer $i\ge 1$.
Consider a $p$-nilpotent group $G$, for example,
and let $N$ be the normal $p$-complement of $G$,
so that $\NN=\{1, N, G\}$ is a normal $p$-series of $G$.
If $\theta\in\Irr(N)$ is invariant in $G$,
then by definition, each character $\chi\in\Irr(G)$ lying over $\theta$
lies in $\Irr(G,\NN)$.
Since $\chi(1)/\theta(1)$ divides $|G:N|$, which is a $p$-power,
it follows that $\chi\in\Irr_i(G,\NN)$, where $p^i=\chi(1)/\theta(1)$.
In this situation, we can choose the group $G$
such that $\chi_N\neq \theta$, and this forces $i\ge 1$.

In the situation of Theorem \ref{A}, for each $i\ge 0$
there is no canonical bijection from $\Irr_i(G,\NN)$ onto $\Irr_i({\bf N}_G(P),\NN)$
in general. 
If $G$ has odd order, however, a canonical bijection does exist
between $\Irr_{p'}(G)$ and $\Irr_{p'}({\bf N}_G(P))$
(see Theorem 10.9 of \cite{I1973}).
We can establish the same result for any positive integer $i$.

\begin{thmA}\label{B}
Let $G$ be a solvable group of odd order,
and let $P\in\Syl_p(G)$.
Suppose that $\NN$ is a normal $p$-series of $G$.
Then there is a canonical bijection from $\Irr(G,\NN)$ onto $\Irr({\bf N}_G(P),\NN)$.
Also, if $\xi\in\Irr({\bf N}_G(P),\NN)$ is the image of
$\chi\in\Irr(G,\NN)$ under this map, then
$\xi$ is a constituent of $\chi_{{\bf N}_G(P)}$
and $\xi(1)$ divides $\chi(1)$.

Furthermore, for each $i\ge 0$, the bijection
maps $\Irr_i(G,\NN)$ onto $\Irr_i({\bf N}_G(P),\NN)$,
and thus defines a canonical bijection between these two sets.
\end{thmA}

We continue to assume that $G$ is a $p$-solvable group,
$p$ is a prime and $P\in\Syl_p(G)$.
In order to prove our Theorems \ref{A} and \ref{B},
we need to extend the definition of the map $\Psi:\Lin(P)\to\Irr(G)$
defined by Isaacs and Navarro in \cite{IN2001},
where $\Lin(P)$ denotes the set of linear characters of $P$.
We now consider an arbitrary character $\alpha\in\Irr(P)$.
By Corollary 3.15 and Theorem 3.16 of \cite{I2018},
we know that $\alpha$ uniquely determines a pair $(W,\gamma)$,
where $W$ is the unique largest subgroup $G$ to which $\alpha$ extends
and $\gamma$ is the unique $p$-special extension of $\alpha$ to $W$.
(For the definition and properties of $p$-special
(or $p'$-special) characters, see the next section.)
In this situation, we can define a new map, which we still call $\Psi$,
$$\Psi:\Irr(P)\to\Char(G),\;\alpha\mapsto \gamma^G,$$
from the set $\Irr(P)$ of irreducible characters of $P$
into the set $\Char(G)$ of all characters of $G$ by setting $\Psi(\alpha)=\gamma^G$.
For convenience, we say that $\Psi$ is the {\bf standard map} associated with $P$.

The following result, which may be of independent interest,
lies at the heart of the proof of Theorems \ref{A} and \ref{B}.

\begin{thmA}\label{C}
Let $P$ be a Sylow $p$-subgroup of a $p$-solvable group $G$, where $p$ is a prime,
and suppose that $\NN$ is a normal $p$-series of $G$.
Let $\Psi:\Irr(P)\to\Char(G)$ be the standard map defined as above.
The following then hold.

{\rm (1)} If $\alpha\in\Irr(P,\NN)$, then $\Psi(\alpha)\in\Irr(G)$.

{\rm (2)} If $\alpha,\beta\in\Irr(P,\NN)$,
then $\Psi(\alpha)=\Psi(\beta)$ if and only if $\alpha$ and $\beta$ are ${\bf N}_G(P)$-conjugate.

{\rm (3)} For each $\alpha\in\Irr(P,\NN)$,
let $W$ be the unique largest subgroup of $G$
to which $\alpha$ extends and let $\gamma$ be the unique $p$-special
extension of $\alpha$ to $W$. Then the map $\delta\mapsto(\gamma\delta)^G$
is an injection from the set of $p'$-special characters of $W$
into $\Irr(G)$.
Write $\Irr_\alpha(G)$ for the image of this injection, i.e.,
$$\Irr_\alpha(G)=\{(\gamma\delta)^G\,|\,
\delta \;\text{is a $p'$-special character of}\; W\},$$
so that $\Irr_\alpha(G)\subseteq\Irr(G)$.
Also, if $\chi\in\Irr_\alpha(G)$, then $\chi(1)_p=\alpha(1)$.

{\rm (4)} If $\alpha,\beta\in\Irr(P,\NN)$,
then $\Irr_\alpha(G)=\Irr_\beta(G)$ if and only if
$\alpha$ and $\beta$ are ${\bf N}_G(P)$-conjugate,
and otherwise, $\Irr_\alpha(G)$ and $\Irr_\beta(G)$ are disjoint.

{\rm (5)}  If $\alpha\in\Irr(P,\NN)$,
then $|\Irr_\alpha(G)|=|\Irr({\bf N}_G(P)|\alpha)|$.

{\rm (6)} $\Irr(G,\NN)=\bigcup\Irr_\alpha(G)$,
where $\alpha$ runs over $\Irr(P,\NN)$.

{\rm (7)} $\Irr({\bf N}_G(P),\NN)=\bigcup\Irr({\bf N}_G(P)|\alpha)$,
where $\alpha$ runs over $\Irr(P,\NN)$.
\end{thmA}

Throughout the paper, all groups considered are finite,
and the notation and terminology are mostly taken from
\cite{I1976} and \cite{I2018}.
In Section 2, we will briefly review some properties of $\pi$-special characters
and the theory of self-stabilizing pairs needed for our proofs,
and in Section 3, we will prove our theorems mentioned above.

\section{Preliminaries}
We begin by introducing some notation.
Suppose that $H$ is a subgroup of a group $G$,
and let $\theta\in\Irr(H)$ be an irreducible character of $H$.
In this situation, we say that $(H,\theta)$ is a {\bf character pair} of $G$.
Observe that $G$ acts by conjugation on the set of character pairs of it,
where by definition $(H,\theta)^g=(H^g,\theta^g)$
and $\theta^g\in\Irr(H^g)$ is defined by the formula $\theta^g(x^g)=\theta(x)$
for $x\in H$. We use ${\bf N}_G(H,\theta)$ to denote the stabilizer of $(H,\theta)$ in $G$.
Following \cite{I2018}, we define a partial order on the set of character pairs of $G$
by setting $(H,\theta)\le (K,\psi)$ if $H\le K$ and $\theta$ lies under $\psi$.

For the reader's convenience, we briefly review the theory of $\pi$-special characters
from \cite{I2018}. Let $G$ be a $\pi$-separable group,
where $\pi$ is a set of primes, and let $\chi\in\Irr(G)$.
We say that $\chi$ is {\bf $\pi$-special} if $\chi(1)$ is a $\pi$-number and
the determinantal order $o(\theta)$ is a $\pi$-number for every irreducible constituent $\theta$ of the restriction $\chi_S$ for every subnormal subgroup $S$ of $G$.
The set of $\pi$-special characters of $G$ is denoted $\X(G)$.

We need some basic properties of $\pi$-special characters.

\begin{lem}\label{prod}
Let $G$ be $\pi$-separable,
and suppose that $\alpha,\beta\in\Irr(G)$ are $\pi$-special and $\pi'$-special, respectively.
Then $\alpha\beta$ is irreducible.
Also, if $\alpha\beta=\alpha'\beta'$, where $\alpha'$ is $\pi$-special and $\beta'$ is $\pi'$-special, then $\alpha=\alpha'$ and $\beta=\beta'$.
\end{lem}
\begin{proof}
This is Theorem 2.2 of \cite{I2018}.
\end{proof}

\begin{lem}\label{res}
Let $G$ be $\pi$-separable, and let $H\le G$ have $\pi'$-index.
Then restriction $\chi\mapsto\chi_H$ defines an injection from $\X(G)$ into $\X(H)$.
\end{lem}
\begin{proof}
See Theorem 2.10 of \cite{I2018}.
\end{proof}

\begin{lem}\label{normal}
Let $N\NM G$, where $G$ is a $\pi$-separable group,
and suppose that $\chi\in\Irr(G)$ and $\theta\in\Irr(N)$.
Then the following hold.

{\rm (1)} If $\chi$ is $\pi$-special, then every irreducible constituent of $\chi_N$ is also $\pi$-special.

{\rm (2)} If $G/N$ is a $\pi$-group, then every member of $\Irr(G|\theta)$
is $\pi$-special.

{\rm (3)} If $G/N$ is a $\pi'$-group and $\theta$ is invariant in $G$,
then $\theta$ has a unique $\pi$-special extension to $G$.
\end{lem}

\begin{proof}
Part (1) follows by definition, and for (2) and (3), see Theorem 2.4 of \cite{I2018}.
\end{proof}

\begin{lem}\label{ext}
Let $H$ be a Hall $\pi$-subgroup of a $\pi$-separable group $G$,
and let $H\le T\le G$. Assume that $\tau\in\Irr(T)$ is $\pi$-special
and $\tau^G$ is irreducible. Write $\beta=\tau_H$.
Then $\beta\in\Irr(H)$ and $T$ is the unique largest subgroup of $G$
to which $\beta$ extends and $\tau$ is the unique $\pi$-special extension of
$\beta$ to $T$.
\end{lem}
\begin{proof}
Since $|T:H|$ is a $\pi'$-number, it follows by Lemma \ref{res} that $\beta$ is irreducible.
Suppose that $W$ is the largest subgroup of $G$ to which $\beta$ extends
and let $\gamma$ be the unique $\pi$-special extension of $\beta$ to $W$.
Then $T\le W$, and by Lemma \ref{res} again, we see that $\gamma_T$ is also $\pi$-special,
and since both $\gamma_T$ and $\tau$ are extensions of $\beta$,
we can deduce that $\gamma_T=\tau$. This completes the proof.
\end{proof}

Now we collect some facts of the theory of self-stabilizing pairs needed for our purpose.
Fix a set $\pi$ of primes, and let $G$ be a $\pi$-separable group.
By a {\bf normal $\pi$-series} $\NN=\{N_i\}$ of length $n$ of $G$, we mean that
a series of normal subgroups
$1=N_0\le N_1\le \cdots \le N_n=G$,
such that $N_{i+1}/N_i$ is either a $\pi$-group or a $\pi'$-group for $i=0,1,\ldots,n-1$.

\begin{defn}\label{ssp}
Let $\mathcal{N}=\{N_i\}$ be a normal $\pi$-series of a $\pi$-separable group $G$,
and let $(T,\tau)$ be a character pair of $G$.
Write $T_i=T\cap N_i$.
Then $(T,\tau)$ is a {\bf self-stabilizing pair} for $\NN$
if for all $i$, the restriction $\tau_{T_i}$ has a unique irreducible constituent $\tau_i$
and $T_{i+1}$ is the stabilizer of the pair $(T_i,\tau_i)$ in $N_{i+1}$.
\end{defn}

Observe that if $(T,\tau)$ is a self-stabilizing pair for $\NN$, then we have the following series of character pairs:
$$(1,1)=(T_0,\tau_0)\le (T_1,\tau_1) \le \cdots \le (T_n,\tau_n)=(T,\tau),$$
which is called a {\bf self-stabilizing chain} for $\NN$.
Writing $\chi=\tau^G$, we see that $\chi$ is irreducible by the following lemma,
and we also say that $(T,\tau)$ is a self-stabilizing pair for $\chi$ with respect to $\NN$.
Furthermore, by Lemmas 3.2 and 3.3 of \cite{L2006},
we know that all the self-stabilizing pairs for $\chi$ with respect to $\NN$ are conjugate in $G$.

\begin{lem}\label{IS}
Let $G$, $\NN$, $(T,\tau)$ and $(T_i,\tau_i)$ be as above.
Then the following hold for all $i$.

{\rm (1)} Induction defines an injection $\Irr({\bf N}_G(T_i,\tau_i)|\tau_i)\to\Irr(G)$.

{\rm (2)} $(\tau_i)^{N_i}$ is irreducible.
\end{lem}

\begin{proof}
This is a special case of Lemma 3.3 of \cite{L2006}.
\end{proof}

Following Lewis \cite{L2006}, we use $\B(G:\mathcal{N})$
to denote the set of characters $\chi\in\Irr(G)$
that have a self-stabilizing pair $(T,\tau)$ in which $\tau$ is $\pi$-special,
and for convenience, we write $\BP(G:\NN)$ for those characters $\chi\in\B(G:\NN)$
such that $|G:T|$ is a $\pi'$-number.

\begin{lem}\label{sat}
Let $\NN$ be a normal $\pi$-series of a $\pi$-separable group $G$,
and suppose that $(T,\tau)$ is a self-stabilizing pair for $\NN$ where $\tau$ is $\pi$-special.
Then the map $\mu\mapsto (\tau\mu)^G$ is an injection from the set $\Y(T)$ of $\pi'$-special characters of $T$ into $\Irr(G)$.
\end{lem}
\begin{proof}
This is exactly Corollary 3.6 of \cite{L2010}.
\end{proof}

Finally, we need a deep result of Wolf,
which we apply to present a new proof of the McKay conjecture for $p$-solvable groups.

\begin{lem}\label{Wolf}
Let $H$ be a Hall $\pi'$-subgroup of a $\pi$-separable group $G$.
Then the number of $\pi$-special characters of $G$ is equal to
$|\Irr({\bf N}_G(H)/H)|$.
\end{lem}
\begin{proof}
See Corollary 1.16 of \cite{W1990}.
\end{proof}

\section{Proofs}
In this section, we will prove our theorems in the introduction.
As in \cite{IN2001}, we will work with $\pi$-separable groups instead of $p$-solvable groups.
We begin by introducing some notation.

Fix a set $\pi$ of primes, and let $G$ be a $\pi$-separable group.
Given a normal $\pi$-series $\NN=\{N_i\}$ of length $n$ of $G$,
we say that $\chi\in\Irr(G)$ is {\bf $(\NN,\pi)$-stable}
if one (and hence every) irreducible constituent of the restriction $\chi_{N_i}$ is invariant in $N_{i+1}$ whenever $N_{i+1}/N_i$ is a $\pi$-group,
and we use $\Irr(G,\NN)$ to denote the set of $(\NN,\pi)$-stable irreducible characters of $G$
(with the set $\pi$ of primes understood).
It is easy to see that $\chi\in\Irr(G,\NN)$ if and only if
there exist characters $\theta_i\in\Irr(N_i)$ with
$$(N_0,\theta_0)\le (N_1,\theta_1)\le \cdots \le (N_n,\theta_n)=(G,\chi),$$
and such that each $\theta_i$ is fixed by $N_{i+1}$ whenever $N_{i+1}/N_i$ is a $\pi$-group.

More generally, if $T$ is a subgroup of $G$,
we define the \emph{restriction} of $\mathcal{N}$ to $T$ as
$\NN_T=\{T\cap N_i\}$. It is clear that $\NN_T$ is a normal $\pi$-series of $T$,
and for notational simplicity, we will always use $\Irr(T,\NN)$ instead of $Irr(T,\NN_T)$.
In particular, if $H$ is a Hall $\pi$-subgroup of $G$ and $\alpha\in\Irr(H)$,
then $\alpha$ is $(\NN_H,\pi)$-stable if and only if $\alpha_{H\cap N}$ is homogeneous
for each $N\in\NN$. Furthermore, if $H\le T$,
then we have $\Irr(H,\NN)=\Irr(H,\NN_T)=\Irr(H,\NN_H)$.

We need the following elementary property of normal $\pi$-series.

\begin{lem}\label{hom}
Let $G$ be a $\pi$-separable group,
$\mathcal{N}=\{N_i\}$ a normal $\pi$-series of length $n$ of $G$,
and $H$ a Hall $\pi$-subgroup of $G$. Suppose that $\alpha\in\Irr(H,\NN)$.
Then the following hold.

{\rm (1)} For any $m\in\{0,1,\ldots,n\}$, let $\alpha_m$ be the unique
irreducible constituent of $\alpha_{H\cap N_m}$ and write
$$\NN_m=\{1 = N_0, N_1, \ldots, N_m \}.$$
Then $\NN_m$ is a normal $\pi$-series of $N_m$
and $\alpha_m\in\Irr(H,\NN_m)$.

{\rm (2)} $\alpha^g\in\Irr(H^g,\NN)$ for any $g\in G$.
\end{lem}
\begin{proof}
It follows immediately from definition.
\end{proof}

The following fundamental result is an essential ingredient in our proofs.

\begin{lem}\label{key}
Let $G$ be $\pi$-separable with Hall $\pi$-subgroup $H$,
and let $\NN$ be a normal $\pi$-series of $G$.
For each $\alpha\in\Irr(H,\NN)$,
let $W$ be the unique largest subgroup of $G$ to which $\alpha$ extends
and let $\gamma$ be the unique $\pi$-special extension of $\alpha$ to $W$.
Then $(W,\gamma)$ is a self-stabilizing pair for $\NN$.
\end{lem}
\begin{proof} We use the notation in Lemma \ref{hom}.
For each $i\in\{0,1,\ldots,n\}$, let $\alpha_i\in\Irr(H\cap N_i)$ be the unique irreducible constituent of $\alpha_{H\cap N_i}$,
and let $W_i=W\cap N_i$, so that $H\cap N_i \le W_i\le N_i$.
Since $\gamma_{H\cap N_i}=(\gamma_H)_{H\cap N_i}=\alpha_{H\cap N_i}$ is a multiple of $\alpha_i$,
it follows that each irreducible constituent of $\gamma_{W_i}$ lies over $\alpha_i$.
But $W_i$ is normal in $W$, so these constituents are all $\pi$-special
(see Lemma \ref{normal}),
and by Lemma \ref{res}, we see that $\gamma_{W_i}$
has a unique irreducible constituent, which we call $\gamma_i$.
In particular, $\gamma_i$ is $\pi$-special and invariant in $W$.
Furthermore, since $H\cap N_i$ is a Hall $\pi$-subgroup of $W_i$,
it follows that $\gamma_i$ is an extension of $\alpha_i$ to $W_i$.
Observe that $W_i\le W_{i+1}$ for each $i$, so $\gamma_i$ lies under $\gamma_{i+1}$.

We claim that $W_i$ is the unique maximal subgroup of $N_i$ to which $\alpha_i$ extends.
To see this, let $U$ be the unique largest subgroup with $H\cap N_i\le U\le N_i$,
and such that $\alpha_i$ extends to $U$.
Then $W_i\le U$ and $\alpha_i$ has a unique $\pi$-special extension $\mu\in\Irr(U)$.
Since $H$ normalizes both $N_i$ and $H\cap N_i$ and fixes $\alpha_i$,
the uniqueness of the pair $(U,\mu)$ implies that
$H$ necessarily normalizes $U$ and fixes $\mu$.
In particular, $UH$ is a subgroup of $G$ with $U\NM UH$,
and $\mu$ is invariant in $UH$.
Note that $H\cap U=H\cap (N_i\cap U)=H\cap N_i$,
and thus Lemma 2.11 of \cite{I2018} guarantees that $\alpha$ extends to $UH$.
By the maximality of $W$, we have $UH\le W$, and hence $U\le W\cap N_i=W_i$.
This proves that $U=W_i$, as claimed.

Now we proceed by induction on $|G|$ to show that $(W,\gamma)$ is
a self-stabilizing pair for $\NN$.
We may assume without loss that $(W_{n-1},\gamma_{n-1})$ is
a self-stabilizing pair for the normal $\pi$-series $\NN_{n-1}$ of
the $\pi$-separable group $N_{n-1}$.
Since we have seen that $\gamma_{W_i}$ is a multiple of $\gamma_i$ for each $i$,
it suffices by definition to show that $W$ is the stabilizer
of the pair $(W_{n-1},\gamma_{n-1})$ in $G$.
For notational convenience, we write $N=N_{n-1}$, $\gamma'=\gamma_{n-1}$ and $\alpha'=\alpha_{n-1}$, so that $W\cap N=W_{n-1}$ and $G/N$
is either a $\pi$-group or a $\pi'$-group.
Let $T={\bf N}_G(W\cap N,\gamma')$ be the stabilizer of $(W\cap N,\gamma')$ in $G$.
Then $W\le T$, and we want to show that $W=T$.

To do this, note that $T\cap N$ is the stabilizer of $\gamma'$ in $N$,
and since $(\gamma')^N$ is irreducible (see Lemma \ref{IS}),
it follows that $T\cap N=W\cap N$.
If $G/N$ is a $\pi$-group, then $TN=G$ because $H\le W\le T$,
and thus $T=W$, as wanted.
So we can assume that $G/N$ is a $\pi'$-group,
and in this case, we have $H\le N$, which implies that
$\alpha=\alpha'$ and that $T/W\cap N$ is a $\pi'$-group.
Since $\gamma'$ is $\pi$-special, it follows
by Lemma \ref{normal} that $\gamma'$ can be extended to $T$,
and thus $\alpha$ also extends to $T$. This forces $T=W$, and
the result follows.
\end{proof}

We now generalize the definition of the map $\Psi$
introduced by Isaacs and Navarro in \cite{IN2001}.
Suppose that $H$ is a Hall $\pi$-subgroup of a $\pi$-separable group $G$,
and let $\alpha\in\Irr(H)$.
By Corollary 3.15 and Theorem 3.16 of \cite{I2018},
there exists a unique largest subgroup $W$ with $H\le W\le G$,
and such that $\alpha$ has a unique $\pi$-special extension $\gamma\in\Irr(W)$.
In this situation, we can define a new map, which we still call $\Psi$,
$$\Psi:\Irr(H)\to\Char(G),\;\alpha\mapsto \gamma^G,$$
from the set $\Irr(H)$ of irreducible characters of $H$
into the set $\Char(G)$ of characters of $G$ by setting $\Psi(\alpha)=\gamma^G$.
For convenience, we say that $\Psi$ is the {\bf standard map} associated with $H$.

The following are some basic properties of the standard map,
which play an important role for our purpose.

\begin{thm}\label{C'}
Let $G$ be a $\pi$-separable group with Hall $\pi$-subgroup $H$,
and let
$$\Psi: \Irr(H)\to\Char(G)$$ be the standard map defined as above.
Suppose that $\NN$ is a normal $\pi$-series of $G$. Then the following hold.

{\rm (1)} If $\alpha\in\Irr(H,\NN)$, then $\Psi(\alpha)\in\Irr(G)$.

{\rm (2)} $\Psi(\Irr(H,\NN))=\BP(G:\NN)$.

{\rm (3)} If $\alpha,\beta\in\Irr(H,\NN)$, then
$\Psi(\alpha)=\Psi(\beta)$ if and only if $\alpha$ and $\beta$ are ${\bf N}_G(H)$-conjugate.

{\rm (4)} For each $\alpha\in\Irr(H,\NN)$,
let $(W,\gamma)$ be as in Lemma \ref{key}.
Then the map $\delta\mapsto(\gamma\delta)^G$
is an injection from the set $\Y(W)$ of $\pi'$-special characters $\delta$ of $W$
into $\Irr(G)$. Write
$$\Irr_\alpha(G)=\{(\gamma\delta)^G\,|\,\delta\in \Y(W)\},$$
so that $\Irr_\alpha(G)\subseteq\Irr(G)$. Also, if $\chi\in\Irr_\alpha(G)$, then $\chi(1)_\pi=\alpha(1)$.

{\rm (5)} Let $\alpha,\beta\in\Irr(H,\NN)$.
If $\alpha$ and $\beta$ are ${\bf N}_G(H)$-conjugate,
then $\Irr_\alpha(G)=\Irr_\beta(G)$,
and otherwise, $\Irr_\alpha(G)$ and $\Irr_\beta(G)$ are disjoint.
\end{thm}

\begin{proof}
By definition, we can write
$$\mathcal{N}=\{1 = N_0 \le N_1 \le \cdots \le N_n = G\},$$
where $N_i$ is normal in $G$ and $N_i/N_{i-1}$ is either a $\pi$-group or a $\pi'$-group for $i=1,\ldots,n$.
For convenience, we let $\alpha$ be an arbitrary but fixed character in $\Irr(H,\NN)$,
and write $\Psi(\alpha)=\gamma^G$,
where $W$ is the unique largest subgroup of $G$ to which $\alpha$ extends
and $\gamma\in\X(W)$ is the unique $\pi$-special extension of $\alpha$ to $W$.

(1) It is immediate by Lemma \ref{key} and Lemma \ref{IS}.

(2) For each $\alpha\in\Irr(H,\NN)$, we see from Lemma \ref{key} that $\Psi(\alpha)\in\BP(G:\NN)$, and thus $\Psi(\Irr(H,\NN))\subseteq \BP(G:\NN)$.
Conversely, let $\chi\in\BP(G:\NN)$.
Then there exists a self-stabilizing pair $(T,\tau)$ for $\NN$ with $\chi=\tau^G$,
and such that $|G:T|$ is a $\pi'$-number and $\tau$ is $\pi$-special.
Replacing $T$ by a conjugate if necessary, we can assume that $H\le T$.
Let $\beta=\tau_H$.
Then $\beta\in\Irr(H)$ and $T$ is the unique maximal subgroup of $G$
to which $\beta$ extends (see Lemma \ref{ext}). Since $\tau_{T\cap N_i}$ is
a multiple of some irreducible character $\tau_i$ of $T\cap N_i$
and $H\cap N_i$ is a Hall $\pi$-subgroup of $T\cap N_i$ for each $i$,
it follows that $(\tau_i)_{H\cap N_i}$ is irreducible and thus $\beta_{H\cap N_i}$
is homogeneous. So $\beta\in\Irr(H,\NN)$, and hence $\Psi(\beta)=\tau^G=\chi$. This establishes the reverse containment, thus proving (2).

(3) Assume first that $\alpha$ and $\beta$ are ${\bf N}_G(H)$-conjugate.
Then $\beta=\alpha^g$ and $H^g=H$ for some $g\in G$.
In this case, it is easy to see that $W^g$ is the largest subgroup to which $\beta$ extends
and that $\gamma^g$ is also the unique $\pi$-special extension of $\beta$ to $W^g$.
By definition, we have $\Psi(\beta)=(\gamma^g)^G=\gamma^G=\Psi(\alpha)$.

Conversely, assume that $\chi=\Psi(\alpha)=\Psi(\beta)$.
Then $\chi=\gamma^G=\nu^G$, where $V$ is the unique maximal
subgroup of $G$ to which $\beta$ extends and $\nu$ is the unique $\pi$-special extension
of $\beta$ to $V$.
By Lemma \ref{key}, we know that both $(W,\gamma)$ and $(V,\nu)$
are self-stabilizing pairs of $\chi$ with respect to $\NN$,
so they are conjugate in $G$. Thus $(V,\nu)=(W,\gamma)^x$ for some $x\in G$.
Now $H\le V=W^x$, so $H^{x^{-1}}$ is also a Hall $\pi$-subgroup of $W$.
It follows that $H^{x^{-1}}=H^w$ for some element $w\in W$,
and thus $y=wx\in{\bf N}_G(H)$ and $(V,\nu)=(W,\gamma)^y$.
This implies that $\nu=\gamma^y$, and
since $\nu_H=\beta$ and $\gamma_H=\alpha$, we have $\beta=\alpha^y$,
as required.

(4) By Lemma \ref{key} again, we see that $(W,\gamma)$ is a self-stabilizing pair for $\Psi(\alpha)$ with respect to $\mathcal{N}$, and
so the first statement follows by Lemma \ref{sat}.
Also, if $\chi\in\Irr_\alpha(G)$, then $\chi=(\gamma\delta)^G$
for some $\pi'$-special character $\delta$ of $W$,
and thus $\chi(1)_\pi=(|G:W|\gamma(1)\delta(1))_\pi=\gamma(1)_\pi=\alpha(1)$.

(5) Suppose that $\alpha$ and $\beta$ are ${\bf N}_G(H)$-conjugate.
Then $\beta=\alpha^g$ for some $g\in G$ and $H^g=H$.
It is easy to see that $W^g$ is the largest subgroup to which $\beta$ extends
and $\gamma^g$ is the unique $\pi$-special extension of $\beta$ to $W^g$.
By definition, we deduce that
$$\Irr_\beta(G)=\{(\gamma^g\delta^g)^G\,|\,\delta^g\in \Y(W^g)\}
=\{(\gamma\delta)^G\,|\,\delta\in \Y(W)\}=\Irr_\alpha(G).$$

Suppose now that there exists some character $\chi\in\Irr_\alpha(G)\cap\Irr_\beta(G)$,
and it suffices to show that $\alpha$ and $\beta$ are ${\bf N}_G(H)$-conjugate.
As we just established, this will imply that $\Irr_\alpha(G)=\Irr_\beta(G)$.
We proceed by induction on the length $n$ of $\NN$ (recall that $N_n=G\in\NN$).
There is nothing to prove if $n=0$ (and hence $G=1$),
so we can assume without loss that the result holds for $N=N_{n-1}$.
(Recall that the corresponding normal $\pi$-series of $N$ is $\NN_{n-1}$ defined in Lemma \ref{hom}.)

As we have seen at the beginning of the proof of Lemma \ref{key},
both $\alpha_{H\cap N}$ and $\gamma_{W\cap N}$ are homogeneous,
and we let $\alpha'$ and $\gamma'$ be the unique irreducible constituents
of $\alpha_{H\cap N}$ and $\gamma_{W\cap N}$, respectively.
Furthermore, we know that $W\cap N$ is the largest subgroup to which $\alpha'$ extends
and that $\gamma'$ is the unique $\pi$-special extension of $\alpha'$ to $W\cap N$.
Now $\chi\in\Irr_\alpha(G)$, so we can write $\chi=(\gamma\delta)^G$ for some $\delta\in\Y(W)$.
Let $\delta'$ be an irreducible constituent of $\delta_{W\cap N}$.
Then $\delta'$ is $\pi'$-special because $W\cap N\NM W$,
and since $\gamma'$ is $\pi$-special,
it follows by Lemma \ref{prod} that $\gamma'\delta'$ is irreducible and lies under $\gamma\delta$.
In particular, $\chi$ lies over $\gamma'\delta'$.
Observer that $(\gamma'\delta')^N\in\Irr(N)$ by (4),
and thus it is an irreducible constituent of $\chi_N$.

Similarly, we can do the same thing for $\chi\in\Irr_\beta(G)$.
Let $V\ge H$ be the largest subgroup of $G$ to which $\beta$ extends,
and let $\nu\in\X(V)$ be the unique $\pi$-special extension of $\beta$ to $V$.
Let $\beta'$ and $\nu'$ be the unique irreducible constituents of $\beta_{H\cap N}$ and $\nu_{V\cap N}$, respectively. Also, we know that $V\cap N$ is the largest subgroup of $N$ to which $\beta'$ extends and that $\nu'$ is the unique $\pi$-special extension of $\beta'$ to $W\cap V$. We can write $\chi=(\nu\mu)^G$ for some $\mu\in\Y(V)$,
and we fix some character $\mu'\in\Y(V\cap N)$ lying under $\mu$.
Then $(\nu'\mu')^N$ is also an irreducible constituent of $\chi_N$.

Now, we have seen that both $(\gamma'\delta')^N$ and $(\nu'\mu')^N$
are irreducible constituents of $\chi_N$, so that they are conjugate in $G$ by Clifford's theorem.
Since $G/N$ is either a $\pi$-group or a $\pi'$-group,
we need to distinguish two possibilities.

Assume first that $G/N$ is a $\pi'$-group.
Then $H\le N$, and hence $G=N\bN_G(H)$ by the Frattini argument.
It follows that $(\gamma'\delta')^N$ and $(\nu'\mu')^N$ can be conjugate by an element of $\bN_G(H)$.
Replacing $\beta$ by an appropriate $\bN_G(H)$-conjugate, we may assume that
$(\gamma'\delta')^N=(\nu'\mu')^N$.
It is easy to see that $\Irr(H,\NN)=\Irr(H,\NN_{n-1})$,
and by the inductive hypothesis applied in $N$,
we deduce that $\alpha$ and $\beta$
are conjugate in $\bN_N(H)$, and thus they are also $\bN_G(H)$-conjugate.
The result follows in this case.

Finally, assume that $G/N$ is a $\pi$-group. Then $NH=G$,
and thus both $W/W\cap N$ and $V/V\cap N$ are $\pi$-groups.
It follows that $\delta_{W\cap N}=\delta'$ and $\mu_{V\cap N}=\mu'$
by Lemma \ref{res}, and in particular, both $\delta'$ and $\mu'$ are $H$-invariant.
Since $\gamma_{W\cap N}$ is a multiple of $\gamma'$
and $\nu_{V\cap N}$ is a multiple of $\nu'$,
we deduce that both $\gamma'$ and $\nu'$ are $W$-invariant.
It follows that both $(\gamma'\delta')^N$ and $(\nu'\mu')^N$
are invariant in $G$, and thus $(\gamma'\delta')^N=(\nu'\mu')^N\in\Irr_{\alpha'}(N)\cap\Irr_{\beta'}(N)$.
By the inductive hypothesis applied in $N$,
we see that $\alpha'$ and $\beta'$ are conjugate in $\bN_N(H\cap N)$,
and hence $\alpha'=(\beta')^x$ for some $x\in \bN_N(H\cap N)$.
Note that the pairs $(W\cap N,\gamma')$
and $(V\cap N, \nu')$ are uniquely determined by $\alpha'$ and $\beta'$, respectively,
and so $(W\cap N,\gamma')=(V\cap N,\nu')^x$.
This implies that $\delta'=(\mu')^x$ by the injection map defined in (4).
Also, since each of $(W,\gamma)$ and $(V,\nu)$ is a self-stabilizing pair for $\NN$
by Lemma \ref{key}, we know that $W$ and $V$ are the stabilizers in $G$ of
the pairs $(W\cap N,\gamma')$ and $(V\cap N,\nu')$, respectively,
and thus $W=V^x$.
Now both $\gamma\delta$ and $\nu^x\mu^x$ are irreducible characters of $W$
that induce $\chi$ and lie over $\gamma'\delta'$,
and since induction defines a bijection
$\Irr(W|\gamma'\delta')\to\Irr(G|(\gamma'\delta')^N)$ (see Lemma 2.12 of \cite{I2018}),
we conclude that $\gamma\delta=\nu^x\mu^x$, which implies that $\gamma=\nu^x$
(see Lemma \ref{prod}).
Observe that $H^x\le V^x=W$, so $H=H^{xw}$ for some $w\in W$,
and we have $\gamma=\gamma^w=\nu^{xw}$.
Recall that $\alpha=\gamma_H$ and $\beta=\nu_H$,
so $\alpha=\beta^{xw}$ with $xw\in\bN_G(H)$,
as required.
The proof is now complete.
\end{proof}

As an application of Theorem \ref{C'}, we prove the following,
which lies at heart of the proof of Theorem \ref{A}.

\begin{thm}\label{A'}
Let $G$ be a $\pi$-separable group with Hall $\pi$-subgroup $H$,
and let $\NN$ be a normal $\pi$-series of $G$.
Then $\chi\in\Irr(G)$ is $(\NN,\pi)$-stable if and only if $\chi\in\Irr_\alpha(G)$
for some $\alpha\in\Irr(H,\NN)$, and thus $\Irr(G,\NN)=\bigcup\Irr_\alpha(G)$,
where $\alpha$ runs over $\Irr(H,\NN)$.
\end{thm}

\begin{proof}
Suppose first that $\chi\in\Irr_\alpha(G)$ for some $\alpha\in\Irr(H,\NN)$.
Then by definition, we can write $\chi=(\gamma\delta)^G$ with $\gamma\in\X(W)$ and $\delta\in\Y(W)$, where $W$ is the unique maximal subgroup of $G$ to which $\alpha$ extends and $\gamma$ is the unique $\pi$-special extension of $\alpha$ to $W$.
By Lemma \ref{key}, we see that $(W,\gamma)$ is a self-stabilizing pair of $\gamma^G$ for $\NN$, and since $|G:W|$ is a $\pi'$-number, we have $\gamma^G\in\BP(G:\NN)$.
Furthermore, let $W_i=W\cap N_i$ for each $i$.
By Definition \ref{ssp}, we know that $\gamma_{W_i}$ has a unique irreducible constituent $\gamma_i$ and that $(W_i,\gamma_i)$ is a self-stabilizing pair for $\NN_i$,
where $\NN_i$ is defined in Lemma \ref{hom}.
It follows by Lemma \ref{IS} that $\psi_i=(\gamma_i)^{N_i}$ is irreducible,
and observe that $|N_i:W_i|=|N_iW:W|$ divides $|G:W|$, which is a $\pi'$-number,
and thus $\psi_i\in\BP(N_i,\NN_i)$.
Furthermore, for each $i$ we can choose some irreducible constituents $\delta_i$ of $\delta_{W_i}$ such that $\delta_i$ lies under $\delta_{i+1}$.
Then we have $(W_i,\delta_i)\le (W_n,\delta_n)=(W,\delta)$.
Since $\delta$ is $\pi'$-special and $W_i\NM W$, it follows that each $\delta_i$ is also $\pi'$-special.

Now we consider the case where $N_{j+1}/N_j$ is a $\pi$-group for some $j$.
Then $N_{j+1}=N_jW_{j+1}$ because $|N_{j+1}:W_{j+1}|$ is a $\pi'$-number.
Since $\gamma_j$ is $W_{j+1}$-invariant and induces $\psi_j$,
it follows that $\psi_j$ is invariant in $N_{j+1}$.
On the other hand, note that $|W_{j+1}:W_j|=|N_{j+1}:N_j|$, which is a $\pi$-number,
and thus $\delta_{j+1}$ restricts irreducibly to $W_j$. This implies that $\delta_{j+1}$ is an extension of $\delta_j$,
and hence $\delta_j$ is $W_{j+1}$-invariant.
Writing $\theta_j=(\gamma_j\delta_j)^{N_j}$, we deduce that $\theta_j$ is fixed by $N_{j+1}$.
By Lemma \ref{sat}, we see that $\theta_j$ is irreducible, and
since $\gamma_j\delta_j$ lies under $\gamma\delta$, it follows that $\theta_j$ lies under $\chi$.
This shows that $\chi$ is $(\NN,\pi)$-stable, as wanted.

Conversely, suppose that $\chi$ is $(\NN,\pi)$-stable.
By definition, there exist characters $\theta_i\in\Irr(N_i)$ for $i=0,1,\ldots,n$, with
$$(N_0,\theta_0)\le (N_1,\theta_1)\le \cdots \le (N_n,\theta_n)=(G,\chi),$$
and such that each $\theta_i$ is invariant in $N_{i+1}$ whenever $N_{i+1}/N_i$ is a $\pi$-group.
Since $\NN_i$ is a normal $\pi$-series of $N_i$ for all $i$, we see that $\theta_i$ is $(\NN_i,\pi)$-stable.
By induction on $n$, we may assume without loss that $\theta_{n-1}\in\Irr_\beta(N_{n-1})$
for some $\beta\in\Irr(H\cap N_{n-1},\NN_{n-1})$.
For notational convenience, we write $N=N_{n-1}$, $\theta=\theta_{n-1}$ and $\NN'=\NN_{n-1}$. Then $\beta\in\Irr(H\cap N,\NN')$,
and we want to show that there exists some $\alpha\in\Irr(H,\NN)$ such that $\chi\in\Irr_\alpha(G)$.

By definition, we can write $\theta=(\nu\mu)^N$ with $\nu\in\X(V)$ and $\mu\in\Y(V)$,
where $V$ is the unique maximal subgroup of $N$ to which $\beta$ extends and $\nu$ is the unique $\pi$-special extension of $\beta$ to $V$.
By Lemma \ref{key} again, we see that $(V,\nu)$ is a self-stabilizing pair of $\nu^N$ for $\NN'$,
and since $|N:V|$ is a $\pi'$-number, we have $\nu^N\in\BP(N:\NN')$.
Let $W={\bf N}_G(V,\nu)$ be the stabilizer of $(V,\nu)$ in $G$.
Then $W\cap N=V$ because $\nu^N$ is irreducible.
Since $G/N$ is either a $\pi$-group or a $\pi'$-group, we have two possibilities.

Assume first that $G/N$ is a $\pi'$-group, so that $H=H\cap N$ is a Hall $\pi$-subgroup of $N$,
and thus $|G:V|$ is a $\pi'$-number. In particular, $W/V$ is also a $\pi'$-group,
and hence $\nu$ can be extended to $W$. By Corollary 3.16 of \cite{I2018},
we see that $\nu$ has a unique $\pi$-special extension $\hat\nu$ to $W$,
and thus by definition $(W,\hat\nu)$ is a self-stabilizing pair for $\NN$.
Writing $\psi=\hat\nu^G$, we have $\psi\in\BP(G:\NN)$.
Furthermore, note that $\chi$ lies over $\theta$ and hence over $\nu\mu$,
so there exists $\tau\in\Irr(W)$ lying under $\chi$ and over $\nu\mu$.
But $(\nu\mu)^W=\hat\nu \mu^W$ and each irreducible constituent of $\mu^W$
is $\pi'$-special, and thus $\tau=\hat\nu\delta$, where $\delta$ is a $\pi'$-special character of $W$. In this case, we know that $\tau^G$ is irreducible by Lemma \ref{sat},
which forces $\chi=(\hat\nu\delta)^G$, and since $H$ is a Hall $\pi$-subgroup of $N$
and $\hat\nu_H=\nu_H=\beta$, it is easy to see that $W$ is exactly the unique maximal subgroup of $G$ to which $\beta$ can be extended (see Lemma \ref{ext}).
Hence $\chi\in\Irr_\beta(G)$
with $\beta\in\Irr(H\cap N,\NN')=\Irr(H,\NN)$.

Next, we suppose that $G/N$ is a $\pi$-group. Then $G=NH$,
and the $(\NN,\pi)$-stability of $\chi$ implies that $\theta$ is invariant in $G$.
We claim that $\nu^N$ is also invariant in $G$.
To see this, fix $h\in H$, and observe that $\theta^h=\theta$
and $\beta^h\in\Irr(H\cap N,\NN')$.
Also, $V^h$ is the unique maximal subgroup of $N$ to which $\beta^h$ extends and
$\nu^h$ is the unique $\pi$-special extension of $\beta^h$ to $V^h$.
By Theorem \ref{C'}(4), with $N$, $\NN'$ and $H\cap N$ in place of $G$,
$\NN$ and $H$, respectively,
we see that $\theta\in\Irr_\beta(N)\cap \Irr_{\beta^h}(N)$.
It follows by Theorem \ref{C'}(5) that $\beta$ and $\beta^h$ are conjugate in ${\bf N}_N(H\cap N)$, and by Theorem \ref{C'}(3), we deduce that $\nu^N=(\nu^h)^N=(\nu^N)^h$.
This shows that $\nu^N$ is $H$-invariant, and hence is invariant in $G$,
as claimed.

Furthermore, since all of the self-stabilizing pairs of $\nu^N$ for $\NN'$ are conjugate in $N$,
we have $G=NW$ by the Frattini argument, and thus $|G:W|=|N:(W\cap N)|=|N:V|$,
which is a $\pi'$-number. Also, since $|W:V|$ divides $|G:N|$,
we see that $W/V$ is a $\pi$-group.
By Lemma \ref{sat} again, we see that $\mu$ is invariant in $W$ because both $\theta$ and $(V,\nu)$
are fixed by $W$, and thus $\mu$ has a unique $\pi'$-special extension $\hat\mu\in\Y(W)$.
Reasoning as before, since $\chi$ lies over $\nu\mu$,
there exists some $\rho\in\Irr(W)$ lying under $\chi$ and over $\nu\mu$,
and observe that $(\nu\mu)^W=\nu^W\hat\mu$ and that each irreducible constituent of $\nu^W$ is $\pi$-special,
so we have $\rho=\gamma\hat\mu$, where $\gamma\in\X(W)$ lies over $\nu$.
Now $(W,\gamma)$ is a self-stabilizing pair for $\NN$ by definition,
and since $|G:W|$ is a $\pi'$-number, we have $\gamma^G\in\BP(G:\NN)$.
This implies that $\chi=\rho^G=(\gamma\hat\mu)^G$.
Furthermore, let $J$ be a Hall $\pi$-subgroup of $W$ containing $H\cap N$ and let $\alpha'=\gamma_J$.
Then $\alpha'\in\Irr(J,\NN)$ and $J$ is also a Hall $\pi$-subgroup of $G$,
and hence $J\cap N=H\cap N$.
Also, since $\gamma$ is a $\pi$-special extension of $\alpha'$ to $W$ and induces irreducibly to $G$,
we see that $W$ is the unique maximal subgroup of $G$ to which $\alpha'$
can be extended, and by definition we have $\chi\in\Irr_{\alpha'}(G)$.
Now $J$ and $H$ are conjugate in $G$, so $H=J^g$ for some $g\in G$.
Write $\alpha=(\alpha')^g$. Then $(J,\alpha')^g=(H,\alpha)$,
and it follows that $\alpha\in\Irr(H,\NN)$.
This shows that $\chi\in\Irr_\alpha(G)$, and the proof is now complete.
\end{proof}

\begin{cor}\label{D}
Let $H$ be a Hall $\pi$-subgroup of a $\pi$-separable group $G$,
and suppose that $\NN$ is a normal $\pi$-series of $G$.
Then the following hold.

{\rm (1)} $|\Irr_\alpha(G)|=|\Irr({\bf N}_G(H)|\alpha)|$ for each $\alpha\in\Irr(H,\NN)$.

{\rm (2)} $\Irr({\bf N}_G(H),\NN)=\bigcup\Irr({\bf N}_G(H)|\alpha)$,
where $\alpha$ runs over $\Irr(H,\NN)$.
\end{cor}

\begin{proof}
(1) Fix $\alpha\in\Irr(H,\NN)$, and as before,
let $W$ be the unique largest subgroup of $G$ to which $\alpha$ extends
and let $\gamma$ be the unique $\pi$-special extension of $\alpha$ to $W$.
Then it is easy to see that $\bN_W(H)$ is also the unique largest subgroup of $\bN_G(H)$ to which $\alpha$ extends and that $\gamma'$ is also the unique $\pi$-special extension of $\alpha$ to $\bN_W(H)$, where $\gamma'$ is the restriction of $\gamma$ to $\bN_W(H)$.
Since $H$ is a normal Hall $\pi$-subgroup of $\bN_G(H)$, it follows that $\bN_W(H)$ is exactly the inertial group of $\alpha$ in $\bN_G(H)$.
By the Clifford correspondence, together with Gallagher's theorem (see Corollary 6.17 of \cite{I1976}), we conclude that
the map $\mu\mapsto(\gamma'\mu)^{{\bf N}_G(H)}$ is a bijection
from $\Irr(\bN_W(H)/H)$ onto $\Irr({\bf N}_G(H)|\alpha)$.
Furthermore, observe that $\bN_W(H)/H$ is a $\pi'$-group and
that $H$ is contained in the kernel of every $\pi'$-special character of $\bN_W(H)$.
Thus $\Irr(\bN_W(H)/H)=\Y(\bN_W(H))$,
and so we obtain a bijection from $\Y(\bN_W(H))$ onto $\Irr({\bf N}_G(H)|\alpha)$.
In particular, we have $|\Y(\bN_W(H))|=|\Irr({\bf N}_G(H)|\alpha)|$.

On the other hand, we obtain $|\Irr_\alpha(G)|=|\Y(W)|$ by Theorem \ref{C'}(4).
Also, by Lemma \ref{Wolf} (with the roles of $\pi$ and $\pi'$ interchanged),
we know that $|\Y(W)|=|\Y(\bN_W(H))|$,
and thus $|\Irr_\alpha(G)|=|\Irr({\bf N}_G(H)|\alpha)|$, as required.

(2) Let $\chi\in\Irr({\bf N}_G(H))$, and let $\alpha\in\Irr(H)$
be a constituent of $\chi_H$.
By definition, it is easy to see that $\chi\in\Irr({\bf N}_G(H),\NN)$
if and only if $\alpha\in \Irr(H,\NN)$,
and the result follows. This completes the proof.
\end{proof}

Now Theorem \ref{C} follows by Theorems \ref{C'} and \ref{A'} and Corollary \ref{D}
(with $p$ in place of $\pi$).

\medskip\noindent\emph{Proof of Theorem \ref{A}.}
By definition, we have $\Irr_{p'}(G)\subseteq \Irr_0(G,\NN)$,
and the reverse containment is clear. Thus (1) follows.
For each integer $i\ge 0$, let $\Delta_i$ be a set of representatives
for the orbits of the conjugation action of ${\bf N}_G(P)$ on $\Irr_i(P,\NN)$.
By Theorem \ref{A'} and Theorem \ref{C'}(5),
we can write
$$\Irr_i(G,\NN)=\bigsqcup_{\alpha\in\Delta_i}\Irr_\alpha(G),$$
and by Corollary \ref{D}, we deduce that
$$\Irr_i({\bf N}_G(P),\NN)=\bigsqcup_{\alpha\in\Delta_i}\Irr({\bf N}_G(P)|\alpha)$$
and that $|\Irr_\alpha(G)|=|\Irr({\bf N}_G(P)|\alpha)|$
for each $\alpha\in\Irr(P,\NN)$.
Thus
$$|\Irr_i(G,\NN)|=|\Irr_i({\bf N}_G(P),\NN)|,$$
and (2) holds. In particular, we obtain
$$|\Irr(G,\NN)|=\sum_{i\ge 0}|\Irr_i(G,\NN)|
=\sum_{i\ge 0}|\Irr_i({\bf N}_G(P),\NN)|=|\Irr({\bf N}_G(P),\NN)|,$$
and (4) follows.

Now we fix an integer $i\ge 0$. For each $\chi\in\Irr(G,\NN)$,
we have seen that $\chi\in\Irr_i(G,\NN)$ if and only if $\chi\in\Irr_\alpha(G)$
for some character $\alpha\in\Irr_i(P,\NN)$,
which is equivalent to saying that
$\chi(1)_p=\alpha(1)=p^i$ by Theorem \ref{C'}(4).
This proves (3), and the proof is complete.
\qed

To show Theorem \ref{B} in the introduction, we need two preliminary results.

\begin{lem}\label{p'-cor}
Let $G$ be a solvable group of odd order, and let $H$ be a Hall $\pi$-subgroup of $G$.
Then there is a canonical bijection from $\Irr_{\pi'}(G)$ onto $\Irr_{\pi'}({\bf N}_G(H))$.
Also, if $\xi$ is the image of $\chi$ under this map, then the following hold.

{\rm (1)} $\xi$ is a constituent of $\chi_{{\bf N}_G(H)}$.

{\rm (2)} $\xi(1)$ divides $\chi(1)$.

{\rm (3)} $\chi$ is $\pi'$-special if and only if $\xi$ is $\pi'$-special.
\end{lem}
\begin{proof}
Parts (1)-(2) follow by Theorem 8.12 of \cite{I2018},
and (3) follows by Corollary 3.3 of \cite{N1993}.
\end{proof}

\begin{lem}\label{T}
Let $\pi$ be a set of primes, let $G$ be a $\pi$-separable group, and
let $K$ be a subgroup of $G$. Then $\nu_\pi(K)$ divides $\nu_\pi(G)$,
where $\nu_\pi(X)$ denotes the number of Hall $\pi$-subgroups of a group $X$.
\end{lem}
\begin{proof}
This is Corollary 1.2 of \cite{T2004}.
\end{proof}

The following result covers Theorem \ref{B} in the introduction
(by taking $\pi=\{p\}$).

\begin{thm}\label{B'}
Let $G$ be a solvable group of odd order, and let $H\in\Hall_\pi(G)$.
Suppose that $\NN$ is a normal $\pi$-series of $G$.
Then there is a canonical bijection from $\Irr(G,\NN)$ onto $\Irr({\bf N}_G(H),\NN)$.
Also, if $\xi\in\Irr({\bf N}_G(H),\NN)$ is the image of
$\chi\in\Irr(G,\NN)$ under this map, then
$\xi$ is a constituent of $\chi_{{\bf N}_G(H)}$
and $\xi(1)$ divides $\chi(1)$.

Furthermore, for each $i\ge 0$, the bijection
maps $\Irr_i(G,\NN)$ onto $\Irr_i({\bf N}_G(H),\NN)$,
and thus defines a canonical bijection between these two sets.
\end{thm}

\begin{proof}
Let $\chi\in\Irr(G,\NN)$.
By Theorem \ref{A'}, we know that $\chi\in\Irr_\alpha(G)$, where $\alpha\in\Irr(H,\NN)$.
Let $W$ be the unique maximal subgroup of $G$ to which $\alpha$ extends and let $\gamma$ be the unique $\pi$-special extension of $\alpha$ to $W$.
Then by Theorem \ref{C'}(4), we have $\chi=(\gamma\delta)^G$
for some $\delta\in \Y(W)$. Let $\gamma'$ be the restriction of $\gamma$ to ${\bf N}_W(H)$,
and note that $|W:{\bf N}_W(H)|$ is a $\pi'$-number,
so $\gamma'$ is also $\pi$-special (see Lemma \ref{res}).
By Lemma \ref{p'-cor}, we have a canonical bijection
from $\Y(W)$ onto $\Y({\bf N}_W(H))$,
and we write $\delta^*$ for the image of $\delta\in\Y(W)$ under this map.
As in the proof of Corollary \ref{D},
the map $\mu\mapsto(\gamma'\mu)^{{\bf N}_G(H)}$ is a bijection
from $\Y(\bN_W(H))$ onto $\Irr({\bf N}_G(H)|\alpha)$.
Let $\tilde\chi=(\gamma'\delta^*)^{{\bf N}_G(H)}$.
Since $\delta^*\in\Y(\bN_W(H))$, it follows that $\tilde\chi\in\Irr({\bf N}_G(H)|\alpha)$.

We claim that $\tilde\chi$ is independent of the choice of $\alpha$.
To see this, assume that $\chi\in\Irr_\beta(G)$, where $\beta\in\Irr(H,\NN)$.
Then $\beta=\alpha^x$ for some element $x\in{\bf N}_G(H)$ by Theorem \ref{C'}(5),
and thus $W^x$ is the unique maximal subgroup of $G$ to which $\beta$ extends and
$\gamma^x$ is the unique $\pi$-special extension of $\beta$ to $W^x$.
Also, $(\gamma')^x$ is the restriction of $\gamma^x$ to ${\bf N}_{W^x}(H)$,
and the canonical bijection from $\Y(W^x)$ onto $\Y({\bf N}_{W^x}(H))$
yields $\delta^x\mapsto(\delta^x)^*=(\delta^*)^x$.
Thus
$$((\gamma')^x(\delta^*)^x)^{{\bf N}_G(H)}=(\gamma'\delta^*)^{{\bf N}_G(H)}
=\tilde\chi,$$
and so $\tilde\chi$ is uniquely determined by $\chi$, as claimed.

Now the map $\chi=(\gamma\delta)^G\mapsto\tilde\chi=(\gamma'\delta^*)^{{\bf N}_G(H)}$ defines a canonical bijection
from $\Irr(G,\NN)$ onto $\Irr({\bf N}_G(H),\NN)$.
By Lemma \ref{p'-cor}, we see that $\delta^*$ lies under $\delta$
and $\delta^*(1)$ divides $\delta(1)$.
It follows that $\gamma'\delta^*$ lies under $\gamma\delta$,
and thus $\tilde\chi$ is a constituent of $\chi_{{\bf N}_G(H)}$.
Note that $\chi(1)=|G:W|\gamma(1)\delta(1)$
and
$$\tilde\chi(1)=|{\bf N}_G(H):{\bf N}_W(H)|\gamma(1)\delta^*(1).$$
Also, by Lemma \ref{T}, we deduce that $|{\bf N}_G(H):{\bf N}_W(H)|$
divides $|G:W|$, and thus $\tilde\chi(1)$ divides $\chi(1)$.

Finally, since $\chi(1)_p=\gamma(1)_p=\tilde\chi(1)_p$,
it follows that our canonical bijection
maps $\Irr_i(G,\NN)$ onto $\Irr_i({\bf N}_G(H),\NN)$.
This completes the proof.
\end{proof}

\section*{Acknowledgements}
The authors wish to thank Professor Gang Chen for his careful reading of the manuscript
and valuable suggestions.
This work was supported by the NSF of China (Nos. 12171289 and 12101374).



\begin{thebibliography}{99}
\bibitem{CS} M. Cabanes, B. Sp\"ath, The McKay conjecture on character degrees, Ann. Math. to appear.

\bibitem{I1973} I.M. Isaacs, Characters of solvable and symplectic groups, Amer. J. Math.
 {\bf 95}  (1973) 594-635.

\bibitem{I1976} I.M. Isaacs, Character Theory of Finite Groups, AMS Chelsea Publishing, 2006.

\bibitem{I2018} I.M. Isaacs, Character of Solvable Group,
                 Amer. Math. Soc, Providence, RI, 2018.

\bibitem{IN2001} I.M. Isaacs, G. Navarro,
    Characters of $p'$-degree of $p$-solvable groups, J. Algebra {\bf 246} (2001) 394-413.

\bibitem{L2006} M.L. Lewis, Obtaining nuclei from chains of normal subgroups,
J. Alg. Appl. {\bf 5} (2006) 215-229.

\bibitem{L2006a} M.L. Lewis, Induction and restriction of lifts of $\pi$-partial characters,
Algebra Colloq. {\bf 13} (2006) 607-616.

\bibitem {L2008} M.L. Lewis, Subnormal inductive sources and $\pi$-partial characters,
Algebra Colloq. {\bf 15} (2008) 405-413.

\bibitem {L2009} M.L. Lewis, Nuclei and lifts of $\pi$-separable groups,
 Algebra Colloq. {\bf 16} (2009) 167-180.

\bibitem{L2010} M.L. Lewis, Lifts of partial characters with respect to a chain of normal subgroups, Alg. Rep. Theory {\bf 13}  (2010) 661-672.

\bibitem{N1993} G. Navarro, Weights, vertices and a correspondence of characters in groups of odd order, Math. Z. {\bf 212} (1993) 535-544.

\bibitem{N2018} G. Navarro, Character Theory and the McKay Conjecture, Cambridge University Press, Cambridge, 2018.

\bibitem{T2004} A. Turull, The number of Hall $\pi$-subgroups of a $\pi$-separable group,
Proc. Amer. Math. Soc.  {\bf 132} (2004) 2563–2565.

\bibitem{W1990} T.R. Wolf, Variations on McKay's character degree conjecture,
J. Algebra {\bf 135} (1990) 123-138.
\end{thebibliography}
\end{document}